\documentclass[12pt, final]{article}
\usepackage{fullpage}
\usepackage{amsthm, amsmath, amssymb, amsfonts}
\usepackage{tikz-cd}
\usepackage{mathtools}
\usepackage{comment}
\usepackage{authblk}
\usepackage[style=alphabetic, giveninits]{biblatex}

\def\R{\mathbb{R}}

\def\A{\mathcal{A}}

\def\Span{\operatorname{Span}}
\def\pr{\operatorname{pr}}
\def\Crit{\operatorname{Crit}}
\def\diag{\operatorname{diag}}
\def\Im{\operatorname{Im}}
\def\ad{\operatorname{ad}}
\def\Ad{\operatorname{Ad}}
\def\grad{\operatorname{grad}}
\def\Hom{\operatorname{Hom}}
\def\End{\operatorname{End}}
\def\Hess{\operatorname{Hess}}
\def\Sp{\operatorname{Sp}}
\def\GL{\operatorname{GL}}

\def\U{\operatorname{U}}
\def\O{\operatorname{O}}
\def\Mat{\operatorname{Mat}}
\def\Gr{\operatorname{Gr}}
\def\Tr{\operatorname{Tr}}
\def\Id{\operatorname{Id}}
\def\sp{\mathfrak{sp}}
\def\so{\mathfrak{so}}
\def\u{\mathfrak{u}}
\def\k{\mathfrak{k}}
\def\p{\mathfrak{p}}

\newtheorem{theorem}{Theorem}[section]
\newtheorem*{theorem*}{Theorem}
\newtheorem{proposition}[theorem]{Proposition}
\newtheorem{remark}[theorem]{Remark}
\newtheorem{definition}[theorem]{Definition}
\newtheorem{lemma}[theorem]{Lemma}
\newtheorem*{example}{Example}
\newtheorem{corollary}[theorem]{Corollary}
\newtheorem*{corollary*}{Corollary}

\newif\ifdebug                                           %
\debugfalse

\usepackage{draftwatermark}

\ifdebug
  \SetWatermarkText{DRAFT}
  \SetWatermarkScale{7}
  \SetWatermarkLightness{.95}
  \PassOptionsToPackage{draft}{showlabels}
\else
  \SetWatermarkText{} 
  \PassOptionsToPackage{final}{showlabels}
\fi

\usepackage{showlabels}
\title{A Morse-Bott unification of the Grassmannians of a symplectic vector space}
\author{Hyunmoon Kim\thanks{Institute of Mathematical Sciences, Ewha Womans University. Email: \texttt{hyunmoon.kim@ewha.ac.kr}}}
\date{\today}

\begin{document}
\maketitle

\begin{abstract}
We construct a quadratic Morse-Bott function on the real Grassmannian of a symplectic vector space from a compatible linear complex structure. We show that its critical loci consist of linear subspaces that split into isotropic and complex parts and that its stable manifolds coincide with the orbits of the linear symplectomorphism group. These orbits generalize the Lagrangian, symplectic, isotropic, and coisotropic Grassmannians to include the Grassmannians of linear subspaces that are neither isotropic, coisotropic, nor symplectic. The negative gradient flow deformation retracts these spaces onto compact homogeneous spaces for the unitary group.
\end{abstract}

\section{Introduction}

The interaction between a symplectic form $\omega$ and a compatible almost complex structure $J$ is central to symplectic topology. For instance, in the study of $J$-holomorphic curves, it is a standard principle that complex submanifolds minimize energy within their homology class.

At the linear level, this relationship is captured by the \textit{Kähler angle} $\theta_W$ of a $2$-plane $W$ (Definition \ref{def:kahler_angle} or see \cite{Scharnhorst} and references therein), which measures the deviation of $W$ from being a complex line. It is well known that $\theta_W = 0$ when $W$ is complex and $\theta_W = \pi/2$ when $W$ is isotropic. While Harvey-Lawson \cite{HarveyLawson} and Donaldson \cite{Donaldson} work with a multiplicative generalization, we present an additive generalization to higher dimensions (Proposition \ref{prop:energy_decomposition}), together with a concrete geometric model via \textbf{Morse-Bott theory} \cite{Bott, AustinBraam}. Our Morse-Bott function assigns a number to \textit{all} linear subspaces of a symplectic vector space according to their compatibility with $J$.

Let $(V, \omega, J)$ be a symplectic vector space with a compatible linear complex structure. For any subspace $W \subset V$, let $P_W$ denote the orthogonal projection onto $W$. We define the energy function on the Grassmannian of $k$-planes $f: \Gr(k; V) \to \R$ by:
\begin{equation}
    f(W) := \frac{1}{2} \Tr([P_W, J]^2).
\end{equation}
This function measures the failure of the subspace to be $J$-invariant. We also interpret this quantity dynamically (Lemma \ref{lem:commutators}) as the kinetic energy of the fundamental vector field $X_J$ generated by the $\U(1)$-action of $J$ on the Grassmannian:
\begin{equation}
    f(W) = \frac{1}{2} \| X_J(W) \|^2.
\end{equation}
This property ensures $f$ is minimized on complex subspaces.

On the other hand, because $f$ involves a commutator with a projection operator, the value of $f$ counts dimensions for subspaces that are suitably ``aligned'' with $J$:

\begin{theorem}[Integrality of critical values, Theorem \ref{thm:morse_bott}, Proposition \ref{prop:critical_values}]
The function $f$ is a Morse-Bott function with integral critical values. Specifically, the energy of a critical subspace is exactly the dimension of its isotropic kernel :
\begin{equation}
    f(W) = \dim_\R(W \cap W^\omega) \quad \text{for all } W \in \Crit(f),
\end{equation}
where $W^\omega$ is the symplectic complement of $W$ (Eq. \eqref{eq:symplectic_complement}).
\end{theorem}

Consequently, the maxima are attained by Lagrangian, isotropic, and coisotropic subspaces (where the isotropic kernel is maximal), and perhaps unsurprisingly, the critical loci consist of the linear subspaces that split into an orthogonal sum of its isotropic kernel with its maximal complex subspace (Proposition \ref{prop:critical_locus}):

\[ W \in \Crit(f) \iff W = (W \cap W^\omega) \oplus (W \cap JW).\]

Analytically, $f$ mimics the norm-squared of a moment map in the sense of \cite{Kirwan} (for linear Morse functions on Grassmannians see \cite{Hangan, TakeuchiFunction, DKV, Guest, BanyagaHurtubise, Nicolaescu}). Although the real Grassmannian is not symplectic in general (for dimensional reasons) and thus admits no classical moment map, $f$ also induces a manifold decomposition invariant under its symmetries. Remarkably, while $f$ is defined via the complex structure $J$, its stable manifolds recover the orbit decomposition of the real Grassmannian $\Gr(k;V)$ under the \textbf{linear symplectomorphism group $\Sp(V)$}.

These orbits are labeled by the invariant indices (Section \ref{subsec:types}) $\vec{n} = (n_0, n_+, n_-)$, which we term the \textbf{type} of the subspace, in analogy with Harald Hess's classification of polarizations (Definition 3.1 of \cite{Hess}). These indices correspond to the dimensions of the isotropic ($n_0$) and reduced ($2n_+$) parts of $W$, and the reduced part of its complement ($2n_-$). We view these orbits as Grassmannian manifolds $\Gr(\vec{n}; V)$, generalizing the Lagrangian ($n_0=n$), isotropic ($n_+ = 0$), coisotropic ($n_- = 0$), and symplectic ($n_0 = 0$) Grassmannians.

\begin{theorem}[Stable manifolds, Theorem \ref{thm:stable_orbits}]
Let $\Gr_J(\vec{n}; V) \subseteq \Gr(k; V)$ be the submanifold of linear subspaces of $V$ of type $\vec{n}$ that split into an orthogonal sum of its isotropic kernel and its maximal complex subspace. Then the stable manifold of $\Gr_J(\vec{n}; V)$ coincides with the Grassmannian $\Gr(\vec{n}; V)$:
\[ W^s(\Gr_J(\vec{n}; V)) = \Gr(\vec{n}; V).\]
\end{theorem}

The negative gradient flow allows us to construct deformation retractions, providing a unified topological description of these orbits.

\begin{corollary}[Deformation retraction, Corollary \ref{cor:def_retraction}]
The Grassmannian $\Gr(\vec{n}; V)$ deformation retracts onto the critical submanifold $\Gr_J(\vec{n}; V)$. Consequently, it is homotopy equivalent to the homogeneous space:
\[
   \Gr(\vec{n}; V) \simeq \Gr_J(\vec{n}; V) \cong \U(n) / (\O(n_0) \times \U(n_+) \times \U(n_-)).
\]
\end{corollary}

This result unifies several topological observations scattered throughout the symplectic literature. Arnold \cite{Arnold} described the Lagrangian Grassmannian as $\U(n)/\O(n)$; Oh-Park \cite{OhPark} described the coisotropic case; and Lee-Leung \cite{LeeLeung} and Ajayi-Banyaga \cite{AjayiBanyaga} analyzed the symplectic case. Our corollary resolves the homotopy type for the ``mixed" cases (neither isotropic, coisotropic, nor symplectic) and constructs all deformation retracts simultaneously via a single Morse-theoretic framework.

We note that this retraction is not easily accessible via standard methods (e.g., \cite[Theorem IV. 3.5]{Loos}). The stabilizer subgroups $\Sp_W(V)$ are nonreductive, as they generally contain nilpotent subgroups (Proposition \ref{prop:stabilizer}), so they cannot be realized as the fixed points of an involution. This nonreductivity limits the use of the combinatorics of root systems and GIT quotients. While Takeuchi \cite{Takeuchi} used negative gradient flows to construct retractions in Hermitian symmetric spaces, that construction was restricted to open orbit strata (where the stabilizer is reductive).

\section*{Organization of paper}

The paper is organized into two main parts: the linear algebraic foundation and the Morse-theoretic analysis.

Section \ref{sec:symplectic_linear_algebra} establishes the classification of linear subspaces in a symplectic vector space. We define the orbit invariants $\vec{n} = (n_0, n_+, n_-)$ and prove the transitivity of the symplectic group action using the relative Darboux theorem. In Sections \ref{subsec:kahler_angles} and \ref{subsec:unitary}, we analyze the geometry of totally real and $J$-compatible subspaces, establishing the crucial link between the symplectic splitting and the K\"{a}hler angle.

Section \ref{sec:morse_bott_theory} develops the Morse-Bott theory on the real Grassmannian. We introduce the energy function $f$ in Section \ref{subsec:energy_function} and prove it is additive with respect to the K\"{a}hler angle decomposition. In Section \ref{subsec:morse_bott}, we compute the Hessian to verify the Morse-Bott non-degeneracy condition. Finally, Sections \ref{subsec:gradient_flow} through \ref{subsec:stable_manifolds} analyze the negative gradient flow. We verify that the flow preserves the isotropic kernel dimension and prove our main result: that the stable manifolds of the energy function coincide exactly with the symplectic orbits $\Gr(\vec{n}; V)$.

\section*{Acknowledgments}
This research was partly funded by the U.S.-Israel Binational Science Foundation (BSF) and by the Natural Sciences and Engineering Research Council of Canada (NSERC), and by the National Research Foundation of Korea(NRF) grant funded by the Korea government(MSIT)(No. RS-2024-00359647). The author would like to thank Yael Karshon and Jae-Hyouk Lee, for their encouragement and many insightful discussions throughout this project.

\section{Symplectic Linear Algebra} \label{sec:symplectic_linear_algebra}
\subsection{Types of linear subspaces}\label{subsec:types}
Let $(V, \omega)$ be a real, $2n$-dimensional symplectic vector space. For any linear subspace $W \subseteq V$, its symplectic complement is given by
\begin{equation} \label{eq:symplectic_complement}
W^\omega = \{ v \in V \mid \omega(v, w) = 0, \text{ for all } w \in W \}.\end{equation} Taking the symplectic complement is an involution, i.e., $(W^\omega)^\omega = W$. The subspaces $W$ and $W^\omega$ have complementary dimensions ($\dim_\R W + \dim_\R W^\omega = \dim_\R V = 2n$) but are not necessarily complementary subspaces in $V$.

The interaction of $W$ with $\omega$ is characterized by its \textbf{isotropic kernel} $W_0 := W \cap W^\omega$. The subspace $W_0$ is an isotropic subspace (contained in its own complement) of $V$. The symplectic complement of $W_0$ in $V$ is the sum $W + W^\omega$. The symplectic form $\omega$ descends to nondegenerate (symplectic) forms on the quotient spaces $W/W_0$ and $W^\omega/W_0$.

We adapt the notation of Harald Hess \cite[Definition 3.1]{Hess}, and classify subspaces by their \textbf{type}, a triple of nonnegative integers $\vec{n} = (n_0, n_+, n_-)$ defined by:
\begin{itemize}
    \item $n_0 = \dim W_0$
    \item $2n_+ = \dim (W/W_0)$
    \item $2n_- = \dim (W^\omega/W_0)$
\end{itemize}
These integers partition the total dimension such that $n = n_0 + n_+ + n_-$. This classification recovers the following conditions: $W$ is \textbf{Lagrangian} if $(n_0, 0, 0)$, \textbf{isotropic} if $(n_0, 0, n_-)$, \textbf{coisotropic} if $(n_0, n_+, 0)$, and \textbf{symplectic} if $(0, n_+, n_-)$. It also gives a systematic way to categorize linear subspaces that are neither coisotropic, isotropic, nor symplectic, according to their behavior with respect to the symplectic form.

\subsection{Transitivity of the linear symplectomorphism group}
Let $\Sp(V)$ be the group of linear symplectomorphisms of $V$. Let $\Gr(\vec{n}; V)$ denote the set of all linear subspaces of type $\vec{n}$.

The transitivity of the $\Sp(V)$ action on $\Gr(\vec{n}; V)$ follows from the linear relative Darboux theorem below.

\begin{definition}[Associated Splittings]\label{def:splittings}
A triple of linear subspaces $(W_+, W_-, W^0)$ of $V$ is called a \textbf{splitting associated to $W$} if:
\begin{enumerate}
    \item $W_+ \subseteq W$ is a symplectic subspace of $V$ complementary to $W_0$ in $W$.
    \item $W_- \subseteq W^\omega$ is a symplectic subspace of $V$ complementary to $W_0$ in $W^\omega$.
    \item $W^0 \subseteq V$ is a Lagrangian subspace complementary to $W_0$ in $(W_+ \oplus W_-)^\omega$.
\end{enumerate}
\end{definition}

\begin{theorem}[Linear relative Darboux theorem {\normalfont(\cite[Section 1.2]{ArnoldGivental})}] \label{thm:rel_darboux}
Let $W \subset V$ be a linear subspace of type $\vec{n}$. Then there exist splittings associated to $W$. For every splitting $(W_+, W_-, W^0)$ associated to $W$, there exists a Darboux basis of $V$:
\[
\{e_1^0, \dots, e_{n_0}^0, e_1^+, \dots, e_{n_+}^+, e_1^-, \dots, e_{n_-}^-, f_1^0, \dots, f_{n_0}^0, f_1^+, \dots, f_{n_+}^+, f_1^-, \dots, f_{n_-}^-\}
\]
adapted to the splitting in the sense that:
\begin{enumerate}
    \item $W_0 = \mathrm{span}\{e_j^0\}_{j=1}^{n_0}$ and $W^0 = \mathrm{span}\{f_j^0\}_{j=1}^{n_0}$.
    \item $W_+ = \mathrm{span}\{e_k^+, f_k^+\}_{k=1}^{n_+}$ and $W_- = \mathrm{span}\{e_l^-, f_l^-\}_{l=1}^{n_-}$.
    \item $W = W_0 \oplus W_+$ and $W^\omega = W_0 \oplus W_-$.
\end{enumerate}
\end{theorem}

As orbits of a connected Lie group, $\Gr(\vec{n}; V)$ is connected for every $\vec{n}$. 

Let $\Gr(k; V)$ denote the ordinary Grassmannian of real $k$-dimensional linear subspaces of $V$. The ordinary Grassmannian $\Gr(k; V)$ decomposes into $\Sp(V)$-orbits as
\[ \Gr(k; V) = \bigsqcup_{\vec{n}: k = n_0 + 2n_+} \Gr(\vec{n}; V).\]

\begin{proposition}[Orbit incidence] \label{prop:orbit_incidence}
$\Gr(\vec{n}'; V)$ is contained in the closure of $\Gr(\vec{n}; V)$ in $\Gr(k; V)$ if and only if $n_0' \ge n_0$.
\end{proposition}
\begin{proof}
We apply the lower semicontinuity of rank on $\omega|_W$. The rank of the restriction can only decrease in the limit, implying the dimension of the isotropic kernel $n_0$ can only increase.
\end{proof}

\subsection{Stabilizer Subgroups and Splittings} \label{subsec:stabilizers_splittings}

Since $g \in \Sp(V)$ is a linear symplectomorphism, it restricts to a linear isomorphism $g|_{W_0}$ on $W_0$ and descends to linear symplectomorphisms $[g]_{W/W_0}$ (resp. $[g]_{W^\omega/W_0}$) on the quotient spaces $W/W_0$ (resp. $W^\omega/W_0$).

Let $\Sp_W(V)\le \Sp(V)$ denote the subgroup of linear symplectomorphisms that stabilize $W$ ($gW \subseteq W$) and consider the restriction-reduction homomorphism:
\[
\Phi: \Sp_W(V) \longrightarrow \GL(W_0) \times \Sp(W/W_0) \times \Sp(W^\omega/W_0)
\]
defined by $\Phi(g) := (g|_{W_0}, [g]_{W/W_0}, [g]_{W^\omega/W_0})$.

Let $H(W) := \ker(\Phi)$. This subgroup contains the information about the shearing of $W_0$ in $W + W^\omega$, as it is a central extension of the abelian group $\Hom_\R(W_0, W+W^\omega)$. It satisfies
\[ H(W) = H(W^\omega) = H(W_0) = H(W + W^\omega).\]

\begin{proposition}
The group $H(W)$ is a connected, simply connected nilpotent Lie group. The exponential map $\exp: \mathfrak{h}(W) \to H(W)$ is a global diffeomorphism, so that
\[
H(W) \cong \R^d, \quad \text{where } d = 2n_0(n_+ + n_-) + \frac{1}{2}n_0(n_0+1).
\]
\end{proposition}
\begin{proof}
We take a Darboux basis from Theorem \ref{thm:rel_darboux}, and reorder the blocks as
\[ \{ e^0, e^+, f^+, e^-, f^-, f^0\}.\]
In this ordered basis, an element $g \in H(W)$ takes the block form
\[
g = \begin{pmatrix}
\Id_{n_0} & E_+ & F_+ & E_- & F_- & Y \\
0 & \Id_{n_+} & 0 & 0 & 0 & F_+^T \\
0 & 0 & \Id_{n_+} & 0 & 0 & -E_+^T \\
0 & 0 & 0 & \Id_{n_-} & 0 & -F_-^T \\
0 & 0 & 0 & 0 & \Id_{n_-} & E_-^T \\
0 & 0 & 0 & 0 & 0 & \Id_{n_0}
\end{pmatrix}
\]
where the symplectic condition imposes that the block $Y$ satisfies $Y - (E_+ F_+^T + E_- F_-^T) \in \Mat_{n_0 \times n_0}(\R)^T$ (i.e., the expression is symmetric).

Since the diagonal blocks are identity matrices and the matrix is block upper-triangular, $H(W)$ is a group of unipotent matrices. Every unipotent matrix group over $\R$ is connected, simply connected, and nilpotent. Furthermore, for unipotent groups, the exponential map $\exp: \mathfrak{h}(W) \to H(W)$ is a global polynomial diffeomorphism from the Lie algebra to the Lie group. Thus, $H(W)$ is diffeomorphic to its Lie algebra, which is a vector space $\R^d$.

We compute the dimension $d$ by counting the independent parameters in the matrix representation:
\begin{itemize}
    \item The blocks $E_+, F_+$ contribute $2 \times (n_0 \times n_+) = 2n_0 n_+$ parameters.
    \item The blocks $E_-, F_-$ contribute $2 \times (n_0 \times n_-) = 2n_0 n_-$ parameters.
    \item The block $Y$ is an $n_0 \times n_0$ matrix. The symplectic constraint fixes the skew-symmetric part of $Y$ in terms of the other blocks, leaving the symmetric part free. The dimension of the space of symmetric $n_0 \times n_0$ matrices is $\frac{1}{2}n_0(n_0+1)$ .
\end{itemize}
Summing these contributions yields the total dimension:
\[
d = 2n_0(n_+ + n_-) + \frac{1}{2}n_0(n_0+1).
\]
\end{proof}

\begin{remark}[Relation to the Heisenberg Group]
In the specific case where $W$ is a Lagrangian subspace (i.e., $n_+ = n_- = 0$), the subgroup $H(W)$ is isomorphic to the abelian group of symmetric matrices $\Mat_{n_0 \times n_0}(\R)^T \cong \R^{n_0(n_0+1)/2}$.

However, for general isotropic subspaces ($n_+ > 0$ or $n_- > 0$), the group $H(W)$ is non-abelian and 2-step nilpotent. It is isomorphic to a \textbf{generalized Heisenberg group} \cite{Ziegler}. Specifically, the center of $H(W)$ corresponds to the symmetric forms on $W_0$, while the quotient by the center acts by shearing the symplectic components $W_\pm$ along $W_0$.
\end{remark}

\begin{proposition}[Structure of the Stabilizer] \label{prop:stabilizer}
Let $W$ be a subspace of type $\vec{n}$. For every splitting $\sigma = (W_+, W_-, W^0)$ associated to $W$, there exists a closed subgroup $G_\sigma \le \Sp_W(V)$ defined by the preservation of the splitting components:
\[
G_\sigma = \{ g \in \Sp_W(V) \mid gW_+ = W_+, gW_- = W_- \}.
\]
The homomorphism $\Phi$ restricts to an isomorphism on $G_\sigma$, yielding the semidirect product decomposition:
\[
\Sp_W(V) \cong G_\sigma \ltimes H(W)
\]
where $G_\sigma \cong \GL(W_0) \times \Sp(W_+) \times \Sp(W_-)$. Moreover, the set of all splittings associated to $W$ is a principal homogeneous space for $H(W)$.
\end{proposition}

\begin{proof}
Let $\sigma = (W_+, W_-, W^0)$ be a splitting associated to $W$. By Theorem \ref{thm:rel_darboux}, there exists a Darboux basis $\mathcal{B}$, reordered as $\{e^0, e^+, f^+, e^-, f^-, f^0\}$, adapted to this splitting. We define $G_\sigma$ as the subgroup of $\Sp_W(V)$ preserving the subspaces $W_+$ and $W_-$. In the basis $\mathcal{B}$, elements of $G_\sigma$ correspond to block-diagonal matrices of the form:
\[
\begin{pmatrix}
X & 0 & 0 & 0 & 0 & 0 \\
0 & A_+ & B_+ & 0 & 0 & 0 \\
0 & C_+ & D_+ & 0 & 0 & 0 \\
0 & 0 & 0 & A_- & B_- & 0 \\
0 & 0 & 0 & C_- & D_- & 0 \\
0 & 0 & 0 & 0 & 0 & (X^{-1})^T
\end{pmatrix}
\]
where $X \in \GL(n_0)$, $\left(\begin{smallmatrix} A_+ & B_+ \\ C_+ & D_+ \end{smallmatrix}\right) \in \Sp(2n_+)$, and $\left(\begin{smallmatrix} A_- & B_- \\ C_- & D_- \end{smallmatrix}\right) \in \Sp(2n_-)$.

The reduction homomorphism $\Phi: \Sp_W(V) \to \GL(W_0) \times \Sp(W/W_0) \times \Sp(W^\omega/W_0)$ restricts to $G_\sigma$. By identifying $\Sp(W/W_0)$ with $\Sp(W_+)$ and $\Sp(W^\omega/W_0)$ with $\Sp(W_-)$ via the splitting, the matrix form above shows that $\Phi|_{G_\sigma}$ is an isomorphism. Since $H(W) = \ker(\Phi)$, the sequence splits, establishing the semidirect product $\Sp_W(V) \cong G_\sigma \ltimes H(W)$.

Finally, we consider the action of $H(W)$ on the set of splittings. Let $\Sigma_W$ be the set of splittings associated to $W$. Given any two splittings $\sigma, \sigma' \in \Sigma_W$, Theorem \ref{thm:rel_darboux} provides adapted Darboux bases for each. The linear map transforming the basis of $\sigma$ to the basis of $\sigma'$ is a symplectic map stabilizing $W$. By factoring out the $G_\sigma$-component (which fixes the splitting $\sigma$), the remaining component lies in $H(W)$. Explicitly, $H(W)$ acts simply transitively on the set of splittings (specifically, on the choices of complements $W_+, W_-, W^0$), making $\Sigma_W$ a torsor for $H(W)$.
\end{proof}

From the definition via the restriction-reduction homomorphism $\Phi$ we can see that $H(W)$ is intrinsic and unique, whereas $G_\sigma$ depends on the geometric choice of splitting.

\begin{remark}[Levi Decomposition]
When $\Sp_W(V)$ and $\Sp(V)$ are viewed as the real points of an algebraic group the decomposition $\Sp_W(V) \cong G_\sigma \ltimes H(W)$ corresponds to the \textbf{Levi decomposition} of a parabolic subgroup $\mathcal{P} = L \ltimes R_u(\mathcal{P})$ (cf. Section 6.4 of \cite{OnishchikVinberg}, Proposition 7.83 of \cite{Knapp2002}). The subgroup $H(W)$ is the unipotent radical $R_u(\mathcal{P})$, and $G_\sigma$ is a Levi factor $L$. The fact that the set of splittings forms a principal homogeneous space for $H(W)$ reflects the conjugacy of Levi factors (the Malcev-Iwasawa theorem).
\end{remark}

\subsection{Totally real symplectic subspaces and K\"{a}hler angles}\label{subsec:kahler_angles}

Let $J$ be an $\omega$-compatible linear complex structure on $V$. The associated inner product is $g(\cdot, \cdot) = \omega(\cdot, J\cdot)$ and the orthogonal complement of a linear subspace $W$ is denoted $W^\perp$. We will refer to a linear subspace $W$ such that $JW = W$ as a \textbf{$J$-invariant} or \textbf{complex} subspace. We will refer to a linear subspace $W$ such that $W \cap JW = \{0\}$ as a \textbf{totally real} subspace.

Using $J$ we can fix a splitting associated to $W$ by letting $W_+:= W_0^\perp \cap W$, $W_-:= W_0^\perp \cap W^\omega$, $W^0 := JW_0$. 
In this section, we will investigate the failure of the splitting
\[ V = (W_0 \oplus JW_0)\oplus W_+ \oplus W_-\]
to be an orthogonal decomposition.

Consider the maximal complex subspaces of $W$ and $W^\omega$, given by $W_+^J := W \cap JW$, and $W_-^J := W^\omega \cap JW^\omega$.
Then for $w_0 \in W_0$ and $w_+ = Jw_+' \in W_+^J$, $g(w_0, w_+) = -\omega(w_0, w_+') = 0$, so $W_\pm^J \subseteq W_\pm$.

Now let $W_\pm^\theta:= (W_\pm^J)^\perp \cap W_\pm$. By construction $W_\pm^\theta$ are transverse to $W_0$, so they are symplectic subspaces of $V$. Then we have the orthogonal decompositions
\begin{align*} W &= W_0 \oplus W_+^J \oplus W_+^\theta, \\ W^\omega &= W_0 \oplus W_-^J \oplus W_-^\theta \end{align*}
into three summands. We will refer to each as the \textbf{isotropic kernel}, the \textbf{$J$-invariant symplectic part}, and the \textbf{totally real symplectic part}. Denote their dimensions $2n_{\pm}^J := \dim_\R W_\pm^J$, and $2n_{\pm}^\theta := \dim_\R W_\pm^\theta$.

\begin{proposition} \label{prop:kahler_splittings}
There exists an orthogonal decomposition of $V$ into $J$-invariant symplectic subspaces
\[ V = V^0 \oplus V^J \oplus V^\theta\]
such that
\begin{enumerate}
    \item $W_0 = W \cap V^0 = W^\omega \cap V^0$ and $W_0$ is a Lagrangian subspace of $V^0$.
    \item $W_+^J = W \cap V^J$ and $W_-^J = W^\omega\cap V^J$. Moreover, they are $J$-invariant symplectic summands of an orthogonal decomposition $V^J = W_+^J \oplus W_-^J$.
    \item $W_+^\theta = W \cap V^\theta$ and $W_-^\theta = W^\omega \cap V^\theta$.  Moreover, they are totally real symplectic summands of a direct sum decomposition $V^\theta = W_+^\theta \oplus W_-^\theta$.
    \item There is an orthogonal decomposition $V^0 = W_0 \oplus JW_0$, and direct sum decompositions $V^\theta = W_\pm^\theta \oplus JW_\pm^\theta$.
\end{enumerate}

\end{proposition}

\begin{proof}
We construct the spaces as $V^0 := W_0 + JW_0$, $V^J := W_+^J + W_-^J$, and $V^\theta := W_+^\theta + W_-^\theta$.
Since $J$ is $\omega$-compatible, the following pairs are orthogonal
\begin{align*} W_0 \perp JW_0, && W_+^J\perp W_-^J, &&V^0 \perp V^J, && W_\pm^\theta \perp W_{\mp}^J.\end{align*}
So we immediately obtain the orthogonal decompositions $V^0 = W_0 \oplus JW_0$, and $V^J = W_+^J \oplus W_-^J$.
Now
\begin{align*}
W_+^\theta \cap W_-^\theta &= (W \cap (W_0 \oplus W_+^J)^\perp) \cap (W^\omega \cap (W_0 \oplus W_-^J)^\perp)\\
&= (W \cap W^\omega) \cap (W_0 \oplus W_+^J \oplus W_-^J)^\perp\\
&\subseteq W_0 \cap W_0^\perp = \{0\}.
\end{align*}
So $V^\theta = W_+^\theta \oplus W_-^\theta$. From $W_\pm^\theta \perp W_{\mp}^J$, we obtain the inclusions $W_\pm^\theta \subseteq (V^0 \oplus V^J)^\perp$ and $V^\theta \subseteq (V^0 \oplus V^J)^\perp$. Moreover, since $W$ and $W^\omega$ have complementary dimensions
\[ \dim_\R V^\theta = 2n_+^\theta + 2n_-^\theta = 2n - 2n_0 - 2n_+^J - 2n_-^J = \dim_\R(V^0 \oplus V^J)^\perp.\]
Therefore $V^\theta = (V^0 \oplus V^J)^\perp$, and is thus $J$-invariant and symplectic.

The subspaces $W_+^\theta$ and $JW_+^\theta$ are transverse by
\[ W_+^\theta \cap JW_+^\theta = (W \cap W_+^\theta) \cap (J(W \cap W_+^\theta)) = (W \cap JW) \cap (W_+^\theta \cap J W_+^\theta) \subseteq V^J \cap V^\theta = \{0\}.\]
and similarly, $W_-^\theta \cap JW_-^\theta = \{0\}$. Since $V^\theta$ is $J$-invariant, $V^\theta$ contains both $W_+^\theta \oplus JW_+^\theta$ and $W_-^\theta \oplus JW_-^\theta$.
Thus $\dim _\R V^\theta = 2n_+^\theta + 2n_-^\theta \ge 2n_\pm^\theta$. Thus $n_+^\theta \ge n_-^\theta$ and $n_+^\theta \le n_-^\theta$, so $n_+^\theta = n_-^\theta$, and we obtain $V = W_\pm^\theta \oplus JW_\pm^\theta$.
\end{proof}

Denote the dimensions of $V^J$ and $V^\theta$ in Proposition \ref{prop:kahler_splittings} by $2n^J :=  \dim_\R V^J$ and $2n^\theta :=\dim_\R V^\theta$. Then we have a suggestive equality
\[ n_0 + n^J + n^\theta = n_0 + n_+ + n_- = n.\]

\begin{corollary}[Dimension constraints] \label{cor:dim_constraints}\
\begin{enumerate}
\item For any linear subspace $W \subseteq V$, the dimensions of the totally real parts of $W$ and $W^\omega$ are equal, i.e. $n_+^\theta = n_-^\theta.$
\item If a symplectic vector space $V$ contains a linear subspace $W$ such that both $W$ and $W^\omega$ are totally real symplectic subspaces of $V$, then the dimension of $V$ is an integer multiple of $4$ and the dimension of $W$ is half the dimension of $V$, i.e. $2n = 4n_+^\theta$.
\end{enumerate}
\end{corollary}
\begin{proof}
This follows from $V^\theta = W_+^\theta\oplus W_-^\theta$ and $V^\theta = W_\pm^\theta \oplus JW_\pm^\theta$. 
\end{proof}

\begin{proposition}[Minimal complex subspaces]\label{prop:minimum_complex}
Let $W\subseteq V$ be a $k$ dimensional linear subspace with $k > n$. Then $W$ contains a complex subspace with complex dimension at least $k-n$.
\end{proposition}
\begin{proof}
We subtract the dimension formulas for $W$ and $W^\omega$:
\begin{align*}
k &= n_0 + 2n_+^J + 2n_+^\theta\\
2n-k &= n_0 + 2n_-^J + 2n_-^\theta.
\end{align*}
Then $k-n = n_+^J - n_-^J$. Since $0 \le n_-^J \le \lfloor (2n-k)/2\rfloor$, we have $n_+^J \ge k-n.$
\end{proof}

Recall the following definition of the K\"{a}hler angle (see \cite{Scharnhorst} and references therein):
\begin{definition}[K\"{a}hler Angle]\label{def:kahler_angle}
Let $W$ be a $2$-dimensional linear subspace of $V$. The \textbf{K\"{a}hler angle} $\theta_W \in [0, \pi/2]$ is the unique angle satisfying
\[
\cos \theta_W = |\omega(u, v)|
\]
for any pair orthonormal vectors of $W$. The angle $\theta_W$ measures the deviation of $W$ from being a complex line ($\theta_W = 0$) or an isotropic plane ($\theta_W = \pi/2$).
\end{definition}

\begin{theorem}[Relative Darboux theorem for totally real symplectic subspaces] \label{thm:totally_real_darboux}
Let $V$ be a symplectic vector space of dimension $2n = 4m$, let $J$ be an $\omega$-compatible linear complex structure and let $W \subseteq V$ be a totally real symplectic linear subspace of dimension $2m$. Then $W^\omega$ is totally real symplectic and there exists a Darboux basis of $V$:
\begin{align*}
\{e_1^+, \dots, e_{k}^+, \quad e_1^{-}, \dots, e_{k}^{-}, \quad f_1^+, \dots, f_{k}^+, \quad f_1^-, \dots, f_{k}^-\}
\end{align*}
and angles $\theta^\pm_1, \dots, \theta^\pm_k \in (0, \pi/2)$ such that:
\begin{enumerate}
    \item $\{ e^+_j, \cos \theta^+_j \cdot f_j^+\}_{j = 1}^k$ is an orthonormal basis of $W$.
    \item $\{e_j^-, \cos \theta^-_j \cdot f^-_j\}_{j = 1}^k$ is an orthonormal basis of $W^\omega$.
\end{enumerate}
\end{theorem}

\begin{proof}
By Corollary \ref{cor:dim_constraints}, $n_+^\theta = n_-^\theta = n$, so $n_-^J = n_0 = 0$. So $W^\omega$ is totally real symplectic. Since $V = W \oplus W^\omega$, it suffices to construct the adapted basis for $W$.

Consider the orthogonal projection $P: V \to W$ and the operator $T: W \to W$ defined by $Tu = P(Ju)$.  The operator $S = -T^2$ is self-adjoint and positive definite. Since $W$ is totally real ($W \cap JW = \{0\}$), $0$ is not an eigenvalue of $S$. Since $W$ is symplectic, $1$ is not an eigenvalue. Thus, the eigenvalues of $S$ lie strictly in $(0, 1)$.

Let $E_j \subseteq W$ be the 2-dimensional eigenspaces of $S$ corresponding to eigenvalues $\lambda_j = (\cos \theta^+_j)^2$. For each $j$, choose a unit vector $e_j^+ \in E_j$. Define $f_j^+$ by:
\[ f_j^+ := (\sec \theta^+_j)^2 T e_j^+. \]
Since $S$ and $T$ commute, $f_j^+$ is also an eigenvector in $E_j \subseteq W$. We can compute  \begin{align*} \omega(e_j^+, f_j^+) =   g(J e_{j}^+, Pf_j^+) = g(PJe_{j}^+, f_j^+)= (\sec \theta^+_j)^2 |Te_j^+|^2 = 1. \end{align*}

Moreover, observe that $g(Tu, v) = g(PJu, v) = g(Ju, v) = \omega(u, v)$ so $T$ is skew-adjoint.
Thus $g(e_j^+, f_j^+) = g(e_j^+, T e_j^+) = 0$. The norm is:
    \[ |f_j^+| = (\sec \theta^+_j)^2 |T e_j^+| = (\sec \theta^+_j)^2 \sqrt{\lambda_j} = \sec \theta^+_j. \]
Thus $\{e_j^+, \cos \theta^+_j \cdot f_j^+\}$ is an orthonormal basis for $E_j$.
\end{proof}

\begin{corollary}[K\"{a}hler Angle Decomposition] \label{cor:kahler_decomp} Let $W \subseteq V$ and $V^\theta$ be as in Proposition \ref{prop:kahler_splittings}. Then there exists an orthogonal decomposition of $V^\theta$ into $J$-invariant symplectic subspaces of real dimension $4$
\[ V^\theta  = \bigoplus_{j = 1}^{n_+^\theta} V^\theta_j\]
such that $W_j^\theta := W \cap V_j^\theta$ is a two dimensional  totally real symplectic linear subspace of $V_j^\theta$ with K\"{a}hler angle in the interval $(0, \pi/2)$.
\end{corollary}

\begin{proof}
$W_+^\theta \subseteq V^\theta$ is a totally real symplectic subspace, that is half the dimension of $V^\theta$. We take the Darboux basis from Theorem \ref{thm:totally_real_darboux}. Then for all $1 \le j \le n_+^\theta$ let $\widetilde{W}_j^\theta := \Span_\R\{ e_j^+, f_j^+\}$ and $V_j^\theta := \widetilde{W}_j^\theta + J\widetilde{W}_j^\theta$. Since $W_+^\theta$ is totally real, $V_j^\theta := \widetilde{W}_j^\theta \oplus J\widetilde{W}_j^\theta$. $V_j^\theta$ is symplectic and $J$-invariant, so this gives an orthogonal decomposition of $V^\theta$. By construcion $W_j^\theta =\widetilde{W}_j^\theta$, and thus is a two dimensional totally real symplectic subspace with K\"{a}hler angle in $(0, \pi/2)$.
\end{proof}

\subsection{$J$-compatible subspaces} \label{subsec:unitary}

\begin{definition}[$J$-Compatibility]
A linear subspace $W \subseteq V$ of type $\vec{n}$ is called \textbf{$J$-compatible} if it is a direct sum of its isotropic kernel $W_0 = W \cap W^\omega$ and its maximal complex subspace $W_+^J = W \cap JW$.
\end{definition}

Let $\U(V, J) \subseteq \Sp(V)$ be the subgroup of linear symplectomorphisms of $V$ that commute with $J$. $\U(V, J)$ is isomorphic to $\U(n)$. Moreover, let $\Gr_J(\vec{n}; V) \subseteq \Gr(\vec{n}; V)$ denote the Grassmannian of $J$-compatible linear subspaces. The subgroup $\U(V, J)$ acts transitively on $\Gr_J(\vec{n}; V)$, as follows from the following version of the Darboux theorem:

\begin{theorem}\label{thm:Jcompatible_darboux}
Let $W \subseteq V$ be a $J$-compatible linear subspace. Then there exists an orthonormal Darboux basis of $V$ of the form
\begin{align*}
\{e_1^0, \dots, e_{n_0}^0,\quad&  e_1^+, \dots, e_{n_+}^+,\quad  e_1^-, \dots, e_{n_-}^-,\\ Je_1^0, \dots, Je_{n_0}^0,&\quad   Je_1^+, \dots, Je_{n_+}^+,\quad Je_1^-, \dots, Je_{n_-}^-\}
\end{align*}
such that $W_0 = \Span_\R\{ e_j^0\}_{j = 1}^{n_0}$, and $W_\pm^J = \Span_\R\{e^\pm_{\ell}, Je^\pm_{\ell}\}_{\ell = 1}^{n_\pm}$.
\end{theorem}

\begin{proof}
By Proposition \ref{prop:kahler_splittings}, the vector space $V$ admits the orthogonal decomposition into $J$-invariant symplectic subspaces
\[
V = (W_0 \oplus J W_0) \oplus W_+^J \oplus W_-^J.
\]
Let $\{e_1^0, \dots, e_{n_0}^0\}$ be an orthonormal basis for $W_0$ with respect to the metric $g$. Since $W_0$ is isotropic, $\omega(e_i^0, e_j^0) = 0$. Since $g$ is $J$-invariant, the set $\{Je_1^0, \dots, Je_{n_0}^0\}$ is an orthonormal basis for $JW_0$. The combined set $\{e_j^0, Je_j^0\}_{j=1}^{n_0}$ forms a Darboux basis for the symplectic subspace $W_0 \oplus JW_0$, satisfying $\omega(e_j^0, Je_k^0) = g(e_j^0, e_k^0) = \delta_{jk}$.
The subspaces $W_\pm^J$ are symplectic and $J$-invariant. We choose Lagrangian splittings $L_{\pm} \oplus L_{\pm}'$ of $W_\pm^J$ and orthonormal bases $\{e_1^\pm, \dots, e_{n_\pm}^\pm\}$ of $L_\pm$. Then $\{e_\ell^\pm, Je_\ell^\pm\}_{\ell=1}^{n\pm}$ are orthonormal Darboux bases of $W_\pm^J$.

Since the decomposition is $g$-orthogonal and $J$-invariant, the union of these bases constitutes the desired Darboux basis for $V$.
\end{proof}

Since $\U(V, J)$ is compact and connected, $\Gr_J(\vec{n}; V) \subseteq \Gr(\vec{n}; V)$ is compact and connected.

\begin{corollary}
\label{cor:iso_coiso_compactness}
    If a linear subspace $W$ is either isotropic or coisotropic, it is $J$-compatible for any $\omega$-compatible linear complex structure $J$ of $V$. Thus if either $n_+ = 0$ or $n_- = 0$
    \[ \Gr_J(\vec{n}; V) = \Gr(\vec{n}; V).\]
\end{corollary}
\begin{proof}
If $n_{\mp}= 0$, then $n_{\mp}^\theta = 0$, so $n_{\pm}^J = n_{\pm}$.
\end{proof}

\begin{proposition}
For a $J$-compatible linear subspace $W$, the stabilizer of $W$ in $\U(V, J)$ is given by:
\[ \U(V, J) \cap \Sp_W(V) \cong O(W_0) \times U(W_+) \times U(W_-) \cong O(n_0) \times U(n_+) \times U(n_-). \]
\end{proposition}
\begin{proof}
We work in the Darboux basis $\mathcal{B} = \{e^0, e^+, e^-, f^0, f^+, f^-\}$ from Theorem \ref{thm:rel_darboux}. Since $W$ is $J$-compatible, we align the complex structure such that $Je^\bullet = f^\bullet$ and $Jf^\bullet = -e^\bullet$. In this specific basis ordering, $J$ takes the standard form:
\[
J = \begin{pmatrix} 0 & -\Id \\ \Id & 0 \end{pmatrix}.
\]
An element $g \in \Sp_W(V)$ stabilizes $W$ and lies in $\U(V, J)$ if and only if it commutes with $J$. Any such $g$ can be uniquely factored as $g = s \cdot h$, where $s$ is in the Levi factor $G_\sigma$ and $h$ is in the unipotent radical $H(W)$.

Using the specific basis order $\mathcal{B}$, an element $h \in H(W)$ has the following explicit $6 \times 6$ block form:
\[
h = \begin{pmatrix}
\Id_{n_0} & E_+ & E_- & Y & F_+ & F_- \\
0 & \Id_{n_+} & 0 & F_+^T & 0 & 0 \\
0 & 0 & \Id_{n_-} & F_-^T & 0 & 0 \\
0 & 0 & 0 & \Id_{n_0} & 0 & 0 \\
0 & 0 & 0 & -E_+^T & \Id_{n_+} & 0 \\
0 & 0 & 0 & -E_-^T & 0 & \Id_{n_-}
\end{pmatrix}.
\]
We compute the commutator with $J = \left(\begin{smallmatrix} 0 & -\Id \\ \Id & 0 \end{smallmatrix}\right)$.
First, consider the product $hJ$:
\[
hJ
=
\begin{pmatrix}
Y & F_+ & F_- & -\Id & -E_+ & -E_- \\
F_+^T & 0 & 0 & 0 & -\Id & 0 \\
F_-^T & 0 & 0 & 0 & 0 & -\Id \\
\Id & 0 & 0 & 0 & 0 & 0 \\
-E_+^T & \Id & 0 & 0 & 0 & 0 \\
-E_-^T & 0 & \Id & 0 & 0 & 0
\end{pmatrix}.
\]
Next, consider the product $Jh$:
\[
Jh=
\begin{pmatrix}
0 & 0 & 0 & -\Id & 0 & 0 \\
0 & 0 & 0 & E_+^T & -\Id & 0 \\
0 & 0 & 0 & E_-^T & 0 & -\Id \\
\Id & E_+ & E_- & Y & F_+ & F_- \\
0 & \Id & 0 & F_+^T & 0 & 0 \\
0 & 0 & \Id & F_-^T & 0 & 0
\end{pmatrix}.
\]
Comparing the $(1,1)$ blocks ($Y$ vs $0$) yields $Y=0$. Comparing the $(1,2)$ blocks ($F_+$ vs $0$) yields $F_+ = 0$. Comparing $(1,5)$ blocks ($-E_+$ vs $0$) yields $E_+ = 0$. The same logic applies to $E_-, F_-$. Thus, $h = \Id$.

Since $h=\Id$, $g = s \in G_\sigma$. In the basis $\mathcal{B}$, $s$ is block-diagonal with respect to the sectors $W_0 \oplus W^0$ (indices 1,4), $W_+$ (indices 2,5), and $W_-$ (indices 3,6) :
\[
s = \begin{pmatrix}
X & 0 & 0 & 0 & 0 & 0 \\
0 & A_+ & 0 & 0 & B_+ & 0 \\
0 & 0 & A_- & 0 & 0 & B_- \\
0 & 0 & 0 & (X^{-1})^T & 0 & 0 \\
0 & C_+ & 0 & 0 & D_+ & 0 \\
0 & 0 & C_- & 0 & 0 & D_-
\end{pmatrix}.
\]
Commuting with $J$ imposes conditions on each $2 \times 2$ block system formed by indices $(k, k+3)$:
\begin{itemize}
    \item The matrix $\left(\begin{smallmatrix} X & 0 \\ 0 & (X^{-1})^T \end{smallmatrix}\right)$ must commute with $\left(\begin{smallmatrix} 0 & -\Id \\ \Id & 0 \end{smallmatrix}\right)$.
    \[
    \begin{pmatrix} 0 & -X \\ (X^{-1})^T & 0 \end{pmatrix} = \begin{pmatrix} 0 & -(X^{-1})^T \\ X & 0 \end{pmatrix} \implies X = (X^{-1})^T \implies XX^T = \Id.
    \]
    Thus $X \in \O(n_0)$.
    \item The matrix $s_+ = \left(\begin{smallmatrix} A_+ & B_+ \\ C_+ & D_+ \end{smallmatrix}\right)$ must commute with $J_+ = \left(\begin{smallmatrix} 0 & -\Id \\ \Id & 0 \end{smallmatrix}\right)$. This implies $s_+$ is complex-linear ($A_+=D_+, B_+=-C_+$). Since $s_+$ is already symplectic, this identifies it as a unitary matrix.
    Thus $s_+ \in \U(n_+)$. Similarly, $s_- \in \U(n_-)$.
\end{itemize}

Combining these results, the stabilizer is isomorphic to the product of these groups:
\[
\U(V, J) \cap \Sp_W(V) \cong \O(n_0) \times \U(n_+) \times \U(n_-). 
\]
\end{proof}

At the linear algebra level choosing $J$ fixes the splitting. At the Lie group level choosing $J$ ignores the unipotent radical and fixes the Levi factor completely.

\section{Morse-Bott theory on the real Grassmannian} \label{sec:morse_bott_theory}

\subsection{Differential geometry of the real Grassmannian} \label{subsec:diffgeo_grassmannian}

We fix an identification $(V, \omega, J) \cong \R^{2n}$. Denote the inner product $g(u, v):= \omega(Ju, v)$ and norm $|u|^2 = g(u, u)$. We identify the real Grassmannian $\Gr(k; V)$ with its image under the embedding $\Gr(k; V) \hookrightarrow \Mat_{2n \times 2n}(\R)$ mapping a linear subspace $W$ to its orthogonal projection matrix $P_W$. The ambient space $\Mat_{2n\times 2n}(\R)$ has an $\O(2n)$-invariant trace inner product $\langle A, B \rangle := \Tr(AB^T)$, and $\Gr(k; V)$ becomes a Riemannian manifold endowed with the restricted trace inner product. We denote the norm $\lVert A \rVert^2:= \langle A, A\rangle$ and the matrix commutator by $[A, B] = \ad_{A}(B)$.

Since $P_W$ is the orthogonal projection matrix it satisfies $P_W^2 = P_W$ and $P_W^T = P_W$. Thus $\Gr(k; V)$ lies inside the linear subspace of symmetric matrices 
\[ \Sigma:= \Mat_{2n \times 2n}(\R)^T \subseteq \Mat_{2n \times 2n}(\R).\] In fact there is a decomposition into disjoint connected components
\[ \{ P \in \Sigma : P^2 = P \} = \bigsqcup_{k = 0}^{2n} \Gr(k; V).\]
Thus we will denote elements in $\Gr(k; V)$ by $W$, $P$, or $P_W$, depending on whether we are viewing them as linear subspaces of $V$, as points in the manifold, or as projection operators associated to linear subspaces.

Consider a curve of projection matrices $P(t)$. Differentiating the identity $P(t)^2 = P(t)$ at $t = 0$ we obtain
\[ P \dot{P}(0) + \dot{P}(0) P = \dot{P}(0).\]
Thus we will identify the tangent space at $P$ as the kernel of the matrix anticommutator with the reflection matrix $\{2P-\Id, \cdot\}$:
\[ T_P \Gr(k; V) \cong \{ X \in \Sigma : \{2P-\Id, X\} = 0\}.\]
In the eigenbasis of $P$ we can also identify $\Hom_\R(W, W^\perp) \cong T_P\Gr(k; V)$ by
\[ Y \mapsto \begin{pmatrix} 0 & Y^T\\ Y & 0 \end{pmatrix}.\]

In particular, we will treat tangent vectors of $\Gr(k; V)$ as \emph{symmetric matrices}.
We identify the linear subspace of $2n \times 2n$ skew-symmetric matrices with the Lie algebra $\so(V)$. The matrix
\[ J:= \begin{pmatrix} 0 & -\Id_{n} \\  \Id_n & 0 \end{pmatrix}\]
is both orthogonal and skew-symmetric, so it can be viewed both as an element of $\O(2n)$ and $\so(V)$. We will abuse notation and denote both by $J$.

We review the geometric meaning of various commutators.
\begin{lemma}\label{lem:commutators}\
    \begin{enumerate}
        \item There exists a short exact sequence
        \[ 0 \to \so(W) \times \so(W^\perp) \to \so(V) \xrightarrow{-\ad_{P_W}} T_{W}\Gr(k; V) \to 0.\]
        Moreover, for any $\xi \in \so(V)$ its fundamental vector field $X_\xi \in \Gamma (T\Gr(k; V))$ is given by
        \[ (X_\xi)_P:= -[P, \xi].\]
        \item Let $\psi_\cdot : \so(V) \to \End (\Gamma(T\Sigma|_{\Gr(k; V)}))$ be defined by
        \[ \psi_\xi(X)_P:= [\widetilde{\xi}_P, X_P]\]
        where $\widetilde{\xi}$ is the constant section of the trivial vector bundle $\Gr(k; V) \times \so(V)$ with value $\xi$. Then $\psi$ is a Lie algebra representation of $\so(V)$.
         \item The orthogonal projection $\pr_P: T_P\Sigma\twoheadrightarrow T_P\Gr(k; V)$ is given by
        \[ \pr_P(A) = \big[P, [P, A]\big].\]
        \item The shape operator $\A_\cdot: T_P\Sigma|_{\Gr(k; V)} \to \End(T_P\Gr(k; V))$ is given by
        \[ \A_X (Y)_P = \big[P, [Y, X]\big] \quad X \in T_P\Sigma|_{\Gr(k; V)}, Y \in T_P\Gr(k; V).\]
        If $X \in T_P \Gr(k; V) \subseteq T_P\Sigma|_{\Gr(k; V)}$ then $\A_X(Y)_P = 0$.
       
    \end{enumerate}
\end{lemma}
\begin{proof}\
    \begin{enumerate}
        \item
The fundamental vector field is obtained by differentiating the conjugation action $P(t) = e^{t\xi} P e^{-t\xi}$ at $t=0$, yielding $X_\xi(P) = \xi P - P \xi = -[P, \xi]$.
Since $\O(V)$ acts transitively on $\Gr(k; V)$, the map $\xi \mapsto X_\xi(P)$ is surjective onto the tangent space. The kernel consists of skew-symmetric matrices commuting with $P$; these are precisely the generators that preserve the eigenspaces $W$ and $W^\perp$, identifying the kernel with $\so(W) \times \so(W^\perp)$.

\item 
The map $\psi_\xi$ acts by the adjoint action $\ad_\xi$. The statement that $\ad$ is a Lie algebra homomorphism is equivalent to the Jacobi identity. Explicitly:
\[ [\psi_\xi, \psi_\eta](X) = [\xi, [\eta, X]] - [\eta, [\xi, X]] = [[\xi, \eta], X] = \psi_{[\xi, \eta]}(X). \]
\item Let $\pr_P(A) = [P, [P, A]]$. Expanding the double commutator yields $\pr_P(A) = PA + AP - 2PAP$.
We first verify the image lies in the tangent space (where $\{P, X\} = X$). Computing the anticommutator:
\begin{align*}
    \{P, \pr_P(A)\} &= P(PA + AP - 2PAP) + (PA + AP - 2PAP)P \\
    &= (PA - PAP) + (AP - PAP) \\
    &= PA + AP - 2PAP = \pr_P(A).
\end{align*}
Next, we characterize the kernel. If $\pr_P(A) = 0$, then $PA + AP = 2PAP$. Multiplying by $P$ on the left gives $PA = PAP$, and on the right gives $AP = PAP$. Thus $PA = AP$, so $[P, A] = 0$.
The kernel consists of matrices commuting with $P$, while the tangent space consists matrices anticommuting with $2P-\Id$). These spaces are orthogonal under the trace inner product, proving $\pr_P$ is the orthogonal projection.
\item 
Let $P(t)$ be a curve in $\Gr(k; V)$ with $P(0)=P$ and $\dot{P}(0)=Y$. Let $X(t)$ be a section of the ambient tangent bundle $T\Sigma|_{\Gr(k; V)}$ along $P(t)$. The shape operator is $\A_X(Y) = -\pr_P(\dot{X}(0))$.

If $X(t)$ is contained in the normal space $N_{P(t)}\Gr(k; V)$ for all $t \in (-\epsilon, \epsilon)$, which is equivalent to $[P(t), X(t)] = 0$. Differentiating at $t=0$:
\[ [\dot{P}, X] + [P, \dot{X}] = [Y, X] + [P, \dot{X}] = 0 \implies [P, \dot{X}] = -[Y, X]. \]
Applying the projection $\pr_P(Z) = [P, [P, Z]]$ to $\dot{X}$:
\[ \pr_P(\dot{X}) = \big[P, [P, \dot{X}]\big] = \big[P, -[Y, X]\big] = -\big[P, [Y, X]\big]. \]
Thus, $\A_X(Y) = -\pr_P(\dot{X}) = \big[P, [Y, X]\big]$.

If $X$ is a tangent vector, let $S = 2P - \Id$. A vector $Z$ is tangent if and only if it anticommutes with $S$ (i.e., $SZ = -ZS$).
We check the action of $S$ on the commutator $\xi = [Y, X]$:
\[ S[Y, X] = [Y, X]S. \]
Since $S$ commutes with $[Y, X]$, $P = \frac{1}{2}(S+\Id)$ also commutes with $[Y, X]$.
Thus, $\A_X(Y) = [P, [Y, X]] = 0$.
    \end{enumerate}
\end{proof}

\subsection{The energy function and the K\"{a}hler angle}\label{subsec:energy_function}

\begin{definition} \label{def:energy}
Let $X_J \in \Gamma(T\Gr(k; V))$ be the vector field defined by
\[ (X_J)_P:= [P, J] = \ad_{P}(J).\]

Let the \textbf{energy function} $f: \Gr(k; V) \to \R$ be defined as follows:
\[
f(P) := \frac{1}{2} \lVert (X_J)_P \rVert^2 = \frac{1}{2} \Tr\left( [P, J] [P, J]^T \right) = \frac{1}{2} \Tr([P, J]^2)
\]
where for $W \in \Gr(k; V)$, $P = P_W$ is the orthogonal projection  operator to $W$.
\end{definition}

\begin{proposition}[Energy decomposition] \label{prop:energy_decomposition} For any $W \in \Gr(\vec{n}; V)$, the energy $f(W)$ is additive with respect to the decomposition in Corollary \ref{cor:kahler_decomp}:
  \[ f(W) = f(W_0) + \sum_{j = 0}^{n_+^\theta} f_j(W_j^\theta) = n_0 + \sum_{j = 1}^{n_+^\theta} 2(\sin \theta_j)^2\]
where $f_j$ is the energy function on $\Gr(2; V_j^\theta)$ and $\theta_{j}$ is the K\"{a}hler angle of $W_j^\theta \subseteq V_j^\theta$.
\end{proposition}

\begin{proof}
Let $W \in \Gr(\vec{n}; V)$. By Corollary \ref{cor:kahler_decomp} and Proposition \ref{prop:kahler_splittings}, the vector space $V$ admits an orthogonal decomposition into symplectic subspaces that are invariant under both the complex structure $J$ and the orthogonal projection $P = P_W$:
\[ V = V^0 \oplus V^J \oplus \bigoplus_{j=1}^{n_+^\theta} V_j^\theta, \]
where $V^0 = W_0 \oplus J W_0$ contains the isotropic kernel, $V^J$ contains the complex part $W_+^J$, and each $V_j^\theta$ is a $4$-dimensional subspace containing a symplectic plane $W_j^\theta$ with K\"{a}hler angle $\theta_j \in (0, \pi/2)$.

Since the decomposition is orthogonal and invariant under $P$ and $J$, the commutator $[P, J]$ is block-diagonal. The energy $f(P) = \frac{1}{2} \Tr([P, J]^2)$ splits as a sum of energies on each block. We compute the contribution of each component using the identity $f(P|_U) = \Tr(P|_U) + \Tr(P|_U J|_U P|_U J|_U)$:

\begin{enumerate}
    \item Restricted to $V^0$, the isotropic kernel $W_0$ is a Lagrangian subspace. In a standard basis adapted to $V^0 \cong \R^{2n_0}$, $P|_{V^0} = \diag(\Id_{n_0}, 0)$ and $J|_{V^0}$ is the standard symplectic matrix. A direct computation shows $[P|_{V^0}, J|_{V^0}]^2 = \Id_{2n_0}$.
    \[ f(P|_{V^0}) = \frac{1}{2} \Tr(\Id_{2n_0}) = n_0. \]

    \item Restricted to $V^J$, the subspace $W \cap V^J$ is $J$-invariant. Thus $P$ and $J$ commute, $[P, J] = 0$, and the energy contribution is $0$.

    \item On each block $V_j^\theta$, let $W_j^\theta = W \cap V_j^\theta$. Since $\dim W_j^\theta = 2$, we have $\Tr(P|_{V_j^\theta}) = 2$.
    Recall from the proof of Theorem \ref{thm:totally_real_darboux} that the K\"{a}hler angle $\theta_j$ is defined via the operator $S = -(P J |_W)^2$, whose eigenvalues on $W_j^\theta$ are $(\cos \theta_j)^2$. Since $P J P J |_{W_j^\theta} = -S$, this operator acts as $-(\cos\theta_j)^2 \Id$. Thus,
    \[ \Tr(P J P J |_{V_j^\theta}) = \Tr(P J P J |_{W_j^\theta}) = -2 (\cos \theta_j)^2. \]
    Using the trace expansion of the energy:
    \begin{align*}
        f_j(W_j^\theta) &= \Tr(P|_{V_j^\theta}) + \Tr(PJPJ|_{V_j^\theta}) \\
        &= 2 - 2 (\cos \theta_j)^2 \\
        &= 2( \sin \theta_j)^2.
    \end{align*}
\end{enumerate}

Summing these contributions yields the result:
\[ f(W) = n_0 + \sum_{j=1}^{n_+^\theta} 2 (\sin \theta_j)^2.  \]
\end{proof}

\begin{lemma}[Energy Bounds]
\label{lem:energy_bounds}
For all $P \in \Gr(\vec{n}; V) \subseteq \Gr(k; V)$,
\[ n_0 \le f(P) \le \min(k, 2n-k).\]
The upper bound is strict if $n_+ > \max(0, k-n)$.
\end{lemma}
\begin{proof}
For the upper bound, consider the energy decomposition on a stratum with $n_+^\theta$ generic pairs:
\[ f(P) = n_0 + \sum_{j=1}^{n_+^\theta} 2 (\sin \theta_j)^2. \]
Since the $0 < \theta_j < \pi/2$, we have the strict inequality $(\sin \theta_j)^2 < 1$. Thus:
\[ f(P) < n_0 + 2n_+^\theta. \]
Using the dimension constraint $k = n_0 + 2n_+^J + 2n_+^\theta$, we substitute $n_0 + 2n_+^\theta = k - 2n_+^J$:
\[ f(P) < k - 2n_+^J. \]
Recall the dimensional constraint that any $k$-dimensional subspace must contain at least $\max(0, k-n)$ complex pairs (Proposition \ref{prop:minimum_complex}).
\begin{itemize}
    \item If $k \le n$, then $n_+^J \ge 0$, so $k - 2n_+^J \le k$.
    \item If $k > n$, then $n_+^J \ge k-n$, so $k - 2n_+^J \le k - 2(k-n) = 2n - k$.
\end{itemize}
Combining these, we obtain the bound:
\[ f(P) < k - 2n_+^J \le \min(k, 2n-k). \]
\end{proof}

\subsection{Critical submanifolds and extrema} \label{subsec:critical}

\begin{lemma}\label{lem:derivatives} \
Let $f: \Gr(k; V) \to \R$ be the energy function defined in Definition \ref{def:energy}.
\begin{enumerate}
    \item The differential of $f$ evaluated at $Y$ is the trace pairing with a double commutator:
    \[ (df)_P(Y) = \langle \psi_J(X_J)_P, Y\rangle =   \Tr\left( \big[ J, [P, J]\big] Y \right). \]
    \item The gradient vector field of $f$ is given by:
    \[ (\grad f)_P = \pr_P (\psi_J(X_J)_P) = \big[ P, [P, [J, [P, J]]] \big]. \]
    \item At any critical point $P \in \Crit(f)$, the Hessian operator $\Hess(f)_P:T_P\Gr(k; V) 
    \to T_P\Gr(k; V)$ is given by:
    \begin{align*}\Hess(f)_P(Y) &= -\pr_P (\psi_J^2(Y)) + \A_{\psi_J(X_J)_P}(Y) \\&= \bigg[P, \big[P, \big [J, [Y, J]\big]\big]\bigg] + \bigg[P, \big[Y, \big[J, [P, J]\big] \big] \bigg].\end{align*}
\end{enumerate}
\end{lemma}
\begin{proof}\
\begin{enumerate}
 \item Let $P(t)$ be a smooth curve in $\Gr(k; V)$ with $P(0) = P$ and $\dot{P}(0) = Y$. Differentiating $f(P) = \frac{1}{2} \Tr([P, J]^2)$:
\begin{align*}
(df)_P(Y) &= \frac{1}{2} \Tr\left( [\dot{P}, J] [P, J] + [P, J] [\dot{P}, J] \right) \\
&= \Tr\left( [Y, J] [P, J] \right) \quad \text{(by cyclic invariance)}.
\end{align*}
Using the identity $\Tr([A, B]C) = \Tr(A [B, C])$, we rewrite the term $[Y, J] = -[J, Y]$:
\[
\Tr( [Y, J] [P, J] ) = \Tr( Y [J, [P, J]] ) = \Tr( [J, [P, J]] Y ).
\]
Identifying $[J, [P, J]]$ as $\psi_J(X_J)_P$, we obtain $(df)_P(Y) = \langle \psi_J(X_J)_P, Y \rangle$.

\item The gradient is the orthogonal projection of the ambient gradient $G = [J, [P, J]]$ onto the tangent space $T_P \Gr(k; V)$. Using the projection operator $\pr_P(A) = [P, [P, A]]$:
\[
(\grad f)_P = \pr_P(G) = \big[ P, [P, [J, [P, J]]] \big].
\]

\item Let $P$ be a critical point, so $(\grad f)_P = 0$, which implies $\pr_P(G) = 0$. Since $\pr_P$ is the orthogonal projection, this means the ambient gradient $G$ is normal to the submanifold.
The Hessian is the covariant derivative of the gradient field:
\[
\Hess(f)_P(Y) = \nabla_Y (\grad f) = \pr_P \left( D_Y (\grad f) \right).
\]
Differentiating the ambient vector field $(\grad f)_Q = \pr_Q(G(Q)) = [Q, [Q, G(Q)]]$ at $P$ in direction $Y$:
\[
D_Y (\grad f) = [Y, [P, G]] + [P, [Y, G]] + [P, [P, D_Y G]].
\]
At a critical point ($[P, G]=0$), the first term vanishes. We project the remaining terms:
\[
\Hess(f)_P(Y) = \pr_P \left( [P, [Y, G]] + [P, [P, D_Y G]] \right).
\]
Using the property that $\pr_P$ is idempotent ($\pr_P \circ \pr_P = \pr_P$) and that $[P, [Y, G]]$ is already a tangent vector (as shown in Lemma \ref{lem:commutators}, shape operator terms are tangent), the projection simplifies:
\[
\Hess(f)_P(Y) = [P, [Y, G]] + \pr_P( D_Y G ).
\]
Substituting $G = \psi_J(X_J)_P$ and $D_Y G = [J, [Y, J]] = -\psi_J^2(Y)$, we obtain
\[
\Hess(f)_P(Y) = -\pr_P (\psi_J^2(Y)) + \A_{\psi_J(X_J)_P}(Y). 
\]
\end{enumerate}
\end{proof}

\begin{proposition}[Critical Locus] \label{prop:critical_locus}
The critical locus of the energy function $f: \Gr(k; V) \to \R$ is  the union the Grassmannians of $J$-compatible linear subspaces:
\[ \Crit(f) = \bigsqcup_{\vec{n}: n_0 + 2n_+ = k} \Gr_J(\vec{n}; V).\]
\end{proposition}

\begin{proof}
Let $P$ be the orthogonal projection onto $W$. We analyze the operators with respect to the orthogonal decomposition $V = W \oplus W^\perp$. In this block splitting, $P = \diag(\Id_W, 0)$. Since $J$ is skew-symmetric, we decompose it as:
\[
J = \begin{pmatrix} A & C \\ -C^T & B \end{pmatrix},
\]
where $A: W \to W$ and $B: W^\perp \to W^\perp$ are skew-symmetric operators, and $C: W^\perp \to W$ is a linear map.
The condition $J^2 = -\Id$ imposes structural constraints on these blocks. Computing the square explicitly:
\[
J^2 = \begin{pmatrix} A^2 - CC^T & AC + CB \\ -C^T A - B C^T & -C^T C + B^2 \end{pmatrix} = \begin{pmatrix} -\Id_W & 0 \\ 0 & -\Id_{W^\perp} \end{pmatrix}.
\]
This yields two identities:
\begin{equation} \label{eq:J2_identities}
C C^T = \Id_W + A^2 \quad \text{and} \quad CB = -AC.
\end{equation}

From Lemma \ref{lem:derivatives}, we compute
\[
(\grad f)_P = 2 \begin{pmatrix} 0 & AC \\ -C^T A & 0 \end{pmatrix}.
\]
Thus, $P$ is a critical point if and only if $AC = 0$.

We claim that if $AC = 0$, then $W$ is $J$-compatible. More specifically, since $A$ is a real skew-symmetric operator on $W$, it is normal and $W$ decomposes orthogonally into $W = \ker(A) \oplus \Im(A)$. We will show $\ker A = W \cap W^\omega$ and $\Im(A)$ is complex.
\begin{enumerate}
    \item Let $v \in \ker(A)$. Then $Av = P(Jv) = 0$, implying $Jv \perp W$, or equivalently $v \perp JW$. Thus $v \in W^\omega$. Hence $\ker(A) = W \cap W^\omega$.
    \item  For any $v \in \Im(A)$, we have $C^T v = 0$ (from the lower right block of $(\grad f)_P$). Substituting this into the identity $CC^T = \Id_W + A^2$:
    \[ 0 = (\Id_W + A^2)v \implies A^2 v = -v. \]
    Since $A = P J|_W$, the condition $A^2 = -\Id$ on $\Im(A)$ implies that $P$ preserves the norm of $Jv$ for all $v \in \Im(A)$, which forces $J(\Im(A)) \subseteq W$.
\end{enumerate}
Thus, $W$ splits orthogonally into an isotropic subspace $\ker(A)$ and a complex subspace $\Im(A)$, proving $W$ is $J$-compatible.
\end{proof}

\begin{proposition}[Critical values]
\label{prop:critical_values}
On each critical submanifold $\Gr_J(\vec{n}; V)$, the critical value of the energy function $f$ is $n_0$.
\end{proposition}
\begin{proof}
This follows from Proposition \ref{prop:energy_decomposition}.
\end{proof}

\begin{corollary}[Extrema]\
\begin{enumerate}
\item If $k \ge n$, the energy function attains its maximum value of $2n-k$ on the Grassmannians of coisotropic linear subspaces. If $k \le n$, the energy function attains its maximum value of $k$ on the Grassmannian of isotropic linear subspaces.
\item If $k$ is even, the energy function attains its minimum value of $0$ on the Grassmannian of $J$-compatible symplectic subspaces (which can be identified with the ordinary complex Grassmannian). If $k$ is odd, the energy function attains its minimum value of $1$ on the Grassmannian of $J$-compatible linear subspaces of type $(1, (k-1)/2, n-(k-1)/2 - 1)$.
\end{enumerate}
\end{corollary}
\begin{proof}
This follows from Proposition \ref{prop:critical_values}.
\end{proof}

\subsection{The Morse-Bott property} \label{subsec:morse_bott}

\begin{theorem}[Morse-Bott Property]\label{thm:morse_bott}
The energy function $f: \Gr(k; V) \to \R$ is a Morse-Bott function. Its critical set is the disjoint union of the smooth submanifolds $\Gr_J(\vec{n}; V)$, and its Hessian is non-degenerate on the normal bundle of these submanifolds.
\end{theorem}

\begin{proof}
The first claim is proved in Proposition~\ref{prop:critical_locus}. 

Fix $P=P_W\in \Gr_J(\vec n;V)$. Since $W$ is $J$-compatible, we have orthogonal decompositions
\[
W=W_0\oplus W_+^J,\qquad W^\perp=JW_0\oplus W_-^J.
\]

\smallskip
Choose an orthonormal basis adapted to the ordered splitting
\[
V = W_0\oplus W_+^J\oplus JW_0\oplus W_-^J.
\]
In this basis,
\[
P=
\begin{pmatrix}
\Id_{n_0} & 0 & 0 & 0\\
0 & \Id_{2n_+} & 0 & 0\\
0 & 0 & 0 & 0\\
0 & 0 & 0 & 0
\end{pmatrix},
\qquad
J=
\begin{pmatrix}
0 & 0 & -\Id_{n_0} & 0\\
0 & J_+ & 0 & 0\\
\Id_{n_0} & 0 & 0 & 0\\
0 & 0 & 0 & J_-
\end{pmatrix},
\]
where $J_\pm^2=-\Id$ and $J_\pm^T=-J_\pm$. 
\smallskip

Using the identification
\begin{align*}
    T_P\Gr(k; V) &= \Hom(W, W^\perp) = \Hom(W_0 \oplus W_+^J, W_-^J\oplus JW_0)\\ &\cong \Hom(W_0, JW_0) \oplus \Hom(W_+^J, JW_0) \oplus \Hom(W_0, \oplus W_-^J) \oplus \Hom(W_+^J, W_-^J),
\end{align*}
we can express a tangent vector $Y \in T_P\Gr(k; V)$ as
\[ Y = (A, B, C, D)\]
where $A\in\Hom(W_0,JW_0)$, $B\in\Hom(W_+^J,JW_0)$, $C\in\Hom(W_0,W_-^J)$, $D\in\Hom(W_+^J,W_-^J)$.
Relative to the splitting $V=W_0\oplus W_+^J\oplus JW_0\oplus W_-^J$,
\[
Y=
\begin{pmatrix}
0 & 0 & A^T & C^T\\
0 & 0 & B^T & D^T\\ 
A & B & 0 & 0\\
C & D & 0 & 0
\end{pmatrix}
\]

\smallskip

In this notation we can compute as in Lemma \ref{lem:derivatives}
\begin{align*}
 -\pr_P(\psi_J^2(Y)) &= [P, [P, [J, [Y, J]]]] = (2(A + A^T), 2B, 2C, 2D + 2J_- D J_+)\\
\A_{\psi_J(X_J)_P}(Y) &= [P, [Y, [J, [P, J]]]] = (-4A, -2B, -2C, 0).\end{align*}

Summing yields the Hessian action:
\begin{equation}\label{eq:hessian-action-correct}
\Hess(f)_P(Y) =  \big(2(A^T-A),\,0,\,0,\,2(D+J_-DJ_+)\big).
\end{equation}

Hence
\begin{equation}\label{eq:ker-hess-correct}
\ker \Hess(f)_P=\{(A,B,C,D):\ A=A^T,\ \ J_-D=DJ_+\},
\end{equation}
and the normal directions are exactly the skew part of $A$ and the $J$-antilinear part of $D$; in particular the Hessian is nondegenerate on the normal bundle.

\smallskip
The critical submanifold containing $P$ is the $\U(V,J)$-orbit passing through $P$, so by Lemma~\ref{lem:commutators}(1) applied to the $\U(V,J)$-action
\[
T_P\Gr_J(\vec n;V)=\ad_P(\u(V,J))=\{-[P,\xi]:\ \xi\in\u(V,J)\}.
\]

The elements of the Lie algebra $\u(V, J) = \sp(V) \cap \so(V)$ are matrices $\xi$ such that $\xi^TJ + J\xi = 0$ and $\xi = -\xi^T$. Equivalently, they are skew-symmetric matrices $X$ that commute with $J$. With respect to the splitting $V = W_0\oplus W_+^J\oplus JW_0\oplus W_-^J$, they are
\[ \xi = \begin{pmatrix} E & BJ_+ & -A^T & -C^T \\ -J_+ B & F & - B^T & -D^T \\ A & B & E & -C^TJ_- \\ C & D & J_-C & G \end{pmatrix} \]
such that the entries satisfy the following:
\begin{itemize}
    \item $E$, $F$, $G$ are skew-symmetric, and $[F, J_+] = 0$ and $[G, J_-] = 0$, i.e.
\[ (E, F, G) \in \so(W_0) \times \u(W_+^J, J_+) \times \u(W_-^J, J_-)\]
\item $A$ is symmetric, and $D$ is complex linear $J_-D = DJ_+$.
\end{itemize}
Computing the commutator, we obtain
\[ [\xi, P] = \begin{pmatrix} 0 & 0 & -A^T & -C^T \\ 0 & 0 & - B^T & -D^T \\ A & B & 0 & 0 \\ C & D & 0 & 0 \end{pmatrix}.\]
and thus we have
\[ \ad_P(\u(V, J)) = T_P\Gr_J(\vec{n}; V) = \ker \Hess(f)_P,\]
which is the Morse-Bott condition.
\end{proof}

\subsection{Gradient flows}
\label{subsec:gradient_flow}

Since $\Gr(k; V)$ is compact, the negative gradient flow $\phi_t$ exists for all $t \in \mathbb{R}$. Moreover, since $f$ is real analytic, the \L ojasiewicz gradient inequality ensures that these flow lines have finite length and each flow line converges to a unique critical point \cite{Lojasiewicz1}, \cite[Section IV.9]{Lojasiewicz2}.

We denote the trajectory by $P(t) := \phi_t(P)$ and its limit by $P_\infty := \lim_{t \to \infty} P(t)$.

\begin{proposition}[Decoupling of negative gradient flow]
\label{prop:decoupling}
The negative gradient flow of $f$ decouples with respect to the  K\"{a}hler angle decomposition in Proposition \ref{cor:kahler_decomp}:
\[ \phi_t(W) = W_0 \oplus W_+^J \oplus \bigoplus_{j = 1}^{n_+^\theta} \phi_t^j(W_j^\theta) \in \Gr(k; V), \]
where $\phi_t^j$ is the negative gradient flow of $f_j$ on $\Gr(2; V_j^\theta)$.
\end{proposition}
\begin{proof}
The gradient vector field of the energy function is given by
\[ (\grad f)_P = \big[ P, [P, [J, [P, J]]] \big]. \]
Let $P(0)$ be the projection onto a subspace $W$ with K\"{a}hler angle decomposition $V = V^0 \oplus V^J \oplus \bigoplus V_j^\theta$.
Within each invariant subspace $U \in \{V^0, V^J, V_j^\theta\}$, the initial projection $P(0)$ acts as a projection $P|_U$ onto $W \cap U$, and the complex structure $J$ restricts to a complex structure $J|_U$.
Since the operations of taking commutators and restricting to invariant subspaces commute, the gradient vector field splits block-diagonally:
\[ (\grad f)_P \big|_U = \big[ P|_U, [P|_U, [J|_U, [P|_U, J|_U]]] \big] = (\grad f|_{U})_{P|_U}. \]

Therefore, the differential equation $\dot{P}(t) = -(\grad f)_{P(t)}$ splits into independent differential equations on each invariant subspace $U$.
Specifically:
\begin{itemize}
    \item On $V^0$, $W_0$ is Lagrangian, so $f|_{V^0}$ is constant (maximal) and $(\grad f|_{V^0}) = 0$. Thus $\phi_t(W_0) = W_0$.
    \item On $V^J$, $W_+^J$ is complex, so $f|_{V^J}$ is constant (minimal) and $(\grad f|_{V^J}) = 0$. Thus $\phi_t(W_+^J) = W_+^J$.
    \item On each $V_j^\theta$, the flow is non-trivial and driven by the restriction of the gradient field to that block.
\end{itemize}
Summing these components gives the result.
\end{proof}

\begin{lemma}
\label{lem:limit_symplectic}
If $P$ is a two dimensional symplectic linear subspace of a four dimensional symplectic vector space $V$, then $f(P_\infty) = 0$.
\end{lemma}

\begin{proof}
We know from Proposition \ref{prop:orbit_incidence}, that because $P_\infty$ is a limit point,
 \[ P_\infty \in \overline{\Gr(0, 1, 1; V)} = \Gr(2, 0, 0; V) \sqcup \Gr(0, 1, 1; V). \]
Moreover, since $P_\infty$ is also a critical point, 
 \[P_\infty \in \Crit(f) \cap \overline{\Gr(0, 1, 1; V)}= \Gr_J(2, 0, 0; V) \sqcup \Gr_J(0, 1, 1; V).\]
  So from Proposition \ref{prop:critical_values} either $f(P_\infty) = 2$ or $f(P_\infty) = 0$.
  
From Lemma \ref{lem:energy_bounds}, and because the negative gradient flow decreases the energy,
\[ f(P_\infty) \le f(P) < 2.\]
So $f(P_\infty) = 0$.
\end{proof}

\begin{example}[Energy on a trajectory]
Consider $V = \R^4$ with symplectic form
\[ \omega(u, v):= v^T \begin{pmatrix} 0 & -\Id_2 \\ \Id_2 & 0 \end{pmatrix} u\]
and complex structure
\[ J = \begin{pmatrix} 0 & -\Id_2 \\ \Id_2 & 0 \end{pmatrix}. \]
Consider the one-parameter subgroup of symplectic matrices $g(t) \in \Sp(V)$ defined by:
\[
g(t) = \begin{pmatrix}
\cosh t & 0 & 0 & \sinh t \\
0 & \cosh t & \sinh t & 0 \\
0 & \sinh t & \cosh t & 0 \\
\sinh t & 0 & 0 & \cosh t
\end{pmatrix}.
\]
Let $W = \Span(e_1, f_1)$ be the standard $J$-invariant reference subspace. Consider the $1$-parameter family $W(t) = g(t)W$. 
Then $W(t)$ is spanned by the orthogonal vectors
\[
v_1 = \cosh t \, e_1 + \sinh t \, f_2, \quad v_2 = \sinh t \, e_2 + \cosh t \, f_1.
\]
The limiting subspace
\[ \lim_{t \to \infty} W(t) = \Span_\R\{ e_1 + f_2, e_2 + f_1\} \]
is Lagrangian.
Normalizing $v_1$ and $v_2$ yields the orthonormal basis $u_1 = v_1/N$ and $u_2 = v_2/N$, where $N = \sqrt{\cosh(2t)}$.
The orthogonal projection matrix $P_{W(t)} = u_1 u_1^T + u_2 u_2^T$ is given by:
\[
P_{W(t)} = \frac{1}{N^2} \begin{pmatrix}
(\cosh t)^2 & 0 & 0 & \cosh t \sinh t \\
0 & (\sinh t)^2 & \sinh t \cosh t & 0 \\
0 & \sinh t \cosh t & (\cosh t)^2 & 0 \\
\cosh t \sinh t & 0 & 0 & (\sinh t)^2
\end{pmatrix}.
\]

The commutator is:
\[
[P_{W(t)}, J] = \frac{\sinh(2t)}{\cosh(2t)} \begin{pmatrix}
0 & 1 & 0 & 0 \\
1 & 0 & 0 & 0 \\
0 & 0 & 0 & -1 \\
0 & 0 & -1 & 0
\end{pmatrix},
\]
and its square is
\[
[P(t), J]^2 = \tanh^2(2t) \cdot \Id_4.
\]
Finally, the energy is the half-trace:
\[
f(W(t)) = \frac{1}{2} \Tr \left( [P(t), J]^2 \right) = 2 \tanh^2(2t).
\]
This trajectory connects the complex minimum ($n_0 = 0$) at $t=0$ to the Lagrangian maximum ($n_0 = 2$) as $t \to \infty$.
\end{example}

\subsection{Type preservation}\label{subsec:type_preservation}

We will view the Lie algebra $\sp(V)$ of $\Sp(V)$ as the matrix Lie algebra \[\sp(V) = \{ X \in \Mat_{2n \times 2n}(\R): XJ + JX^T = 0\}\]

The Lie algebra $\sp(V)$ admits a Cartan decomposition based on the interaction with the complex structure $J$. For $\sp(V)$ the Cartan involution $\theta(X) = - X^T$ agrees with $\Ad(J)(X)$ by rearranging the defining relation $XJ + JX^T = 0$. 

Since $\theta^2 = \Id$, the algebra splits into the $\pm 1$ eigenspaces of $\theta$:
\[ \sp(V) = \k \oplus \p, \]
where $\k$ is the $+1$ eigenspace and $\p$ is the $-1$ eigenspace. The eigenspace condition implies $\k \subseteq \so(V)$ and $\p \subseteq \Sigma$. Because $\sp(V)$ may have non skew-symmetric elements, the conjugation $e^{t\xi} P e^{-t\xi}$ may not be an orthogonal projection. To account for this, let
\[ \ad^\bullet_P (\xi) := \frac{1}{2} \ad_P (\xi - \xi^T) - \frac{1}{2} \ad_P^2 (\xi+\xi^T)\]
for $P \in \Gr(k; V)$ and $\xi \in \sp(V)$.
\begin{proposition}[Fundamental Vector Fields]
\label{prop:fundamental_vf}
There exists a short exact sequence
\[ 0 \to \sp_W(V) \to \sp(V) \xrightarrow{-\ad^\bullet_{P_W}} T_{W}(\Sp(V)\cdot W) \to 0.\]
For any $\xi \in \sp(V)$ its fundamental vector field $X_\xi \in \Gamma(T\Gr(k; V))$ is given by
        \[ (X_\xi)_P:= -\ad^\bullet_P(\xi).\]
\end{proposition}
\begin{proof}
Let $W \in \Gr(k; V)$ with orthogonal projection matrix $P$. Consider the curve of subspaces $W(t) = e^{t\xi} W$ generated by $\xi \in \sp(V)$. Let $P(t)$ be the smooth curve of orthogonal projection matrices onto $W(t)$. We compute the tangent vector $X_\xi(P) := \dot{P}(0)$.

The condition that $P(t)$ projects onto $W(t)$ means that for any $u \in W$, the vector $v(t) = e^{t\xi}u$ satisfies $P(t) v(t) = v(t)$. Differentiating at $t=0$:
\[ \dot{P} P + P \xi P = \xi P. \]
Multiplying by $P^\perp = \Id - P$ on the left annihilates the second term ($P^\perp P = 0$), yielding:
\[ P^\perp \dot{P} P = P^\perp \xi P. \]
Since $P(t)$ is a projection ($P^2=P$), differentiating yields $\dot{P}P + P\dot{P} = \dot{P}$. Multiplying by $P$ on both sides yields $P\dot{P}P = 0$ (diagonal blocks vanish). Thus $\dot{P}$ decomposes into off-diagonal blocks:
\[ \dot{P} = P \dot{P} P^\perp + P^\perp \dot{P} P. \]
Since $P(t)$ is symmetric, $\dot{P}$ is symmetric. Therefore, the first term is the transpose of the second:
\[ P \dot{P} P^\perp = (P^\perp \dot{P} P)^T = (P^\perp \xi P)^T = P \xi^T P^\perp. \]
Summing the two blocks yields
\[ \dot{P} = P^\perp \xi P + P \xi^T P^\perp. \]

Now we utilize the identities $P^\perp A P = \frac{1}{2}([P, [P, A]] + [A, P])$ and $P A P^\perp = \frac{1}{2}([P, [P, A]] - [A, P])$ for any matrix $A$.
Substituting $A=\xi^T$ into the first and $A=\xi$ into the second:
\begin{align*}
P \xi P^\perp &= \frac{1}{2} \big( [P, [P, \xi^T]] + [\xi^T, P] \big) \\
P^\perp \xi P &= \frac{1}{2} \big( [P, [P, \xi]] - [\xi, P] \big)
\end{align*}
Summing these gives :
\[
(X_\xi)_P= \dot{P} = \frac{1}{2} \left[ P, \left[ P, \xi + \xi^T \right] \right] - \frac{1}{2} \left[ P, \xi - \xi^T \right] = -\ad_P^\bullet(\xi).
\]

The kernel of $\xi \mapsto (X_\xi)_P$ is the stabilizer algebra $\sp_W(V)$ (generators fixing $W$). The map is surjective onto the tangent space of the orbit $T_W(\Sp(V)\cdot W)$. Thus the sequence is exact.
\end{proof}

\begin{lemma}[Gradient as an infinitesimal symplectic symmetry]
\label{lem:gradient_symmetry}
For any $P \in \Gr(k; V)$, the negative gradient vector $-(\grad f)_P \in T_P(\Sp(V)\cdot P)$ and indeed equals the fundamental vector field at $P$ for the symmetry generator $Z = -[J, [P, J]] \in \p \subseteq \sp(V)$.
\end{lemma}
\begin{proof}
Let $Z$ be the image of $-\psi_J(X_J)_P \in T\Sigma|_P$ under the isomorphism $T\Sigma|_P \cong \Sigma$. Then as matrices $Z = - [J, [P, J]]$.
We first check that $Z \in \p$.
Since $J$ is skew and $P$ is symmetric, $[P, J]$ is symmetric, and $[J, [P, J]]$ is symmetric. Thus $Z^T = Z$.
Checking the symplectic condition:
\[ \{Z, J\} = -[J, [P, J]]J - J[J, [P, J]] = - (J[P, J]J - [P, J]J^2) - (J^2[P, J] - J[P, J]J) = 0. \]
Thus $Z \in \p$.
Using the formula from Proposition \ref{prop:fundamental_vf}, the fundamental vector field is:
\[ (X_Z)_P = -\ad^\bullet_P(Z) = -(-\ad_P^2(Z)) = [P, [P, Z]]. \]
Substitute $Z = -[J, [P, J]]$:
\[ (X_Z)_P = [P, [P, -[J, [P, J]]]] = - [P, [P, [J, [P, J]]]] = -(\grad f)_P. \]
\end{proof}

\begin{lemma}[Type preservation]
\label{lem:type_preservation}
Let $\{\phi_t\}_{t \in \R}$ be the negative gradient flow for the energy function, and for any point $P$ in $\Gr(k; V)$, let $P(t):= \phi_t(P)$ and $P_\infty := \lim_{t \to \infty} P(t)$. Then
\[ n_0(P) = n_0 (P(t)) = n_0 (P_\infty).\]
\end{lemma}

\begin{proof}
Let $P \in \Gr(\vec{n}; V)$. By Lemma \ref{lem:gradient_symmetry}, the gradient vector field is tangent to the symplectic orbit $\Sp(V) \cdot P$. Therefore, for all finite $t$, the flow $P(t)$ remains in the same orbit, implying $n_0(P(t)) = n_0(P)$.

Moreover, the limit $P_\infty$ must lie in the closure of this orbit (Proposition \ref{prop:orbit_incidence}):
\[ P_\infty \in \Crit (f) \cap \overline{\Gr(\vec{n}; V)} = \bigsqcup_{\substack{\vec{n}': n'_0 + 2n'_+ = k\\ n_0' \ge n_0}} \Gr_J(\vec{n}'; V). \]
Thus $n_0(P_\infty) \ge n_0(P)$.

Since $f(P(t))$ is monotonically decreasing $f(P_\infty) \le f(P(t))$ for all $t$. By Proposition \ref{prop:critical_values}, the left hand side is $f(P_\infty) = n_0(P_\infty)$.

By Proposition \ref{prop:energy_decomposition} and Lemma \ref{lem:limit_symplectic}, the right hand bound sharpens to $n_0(P)$ as $t \to \infty$. (In each four dimensional symplectic vector space $V_j^\theta$ of Proposition \ref{cor:kahler_decomp}, Lemma \ref{lem:limit_symplectic} shows the energy decays as $f_j(W_j(t))  \to 0$; summing over $j$ yields $f(P(t)) \to n_0(P)$.) So $n_0(P_\infty) \le n_0(P)$.
\end{proof}

\subsection{Stable manifolds as symplectic orbits}
\label{subsec:stable_manifolds}

\begin{theorem}
\label{thm:stable_orbits}
The stable manifold of the critical submanifold $\Gr_J(\vec{n}; V)$ coincides with the symplectic orbit of type $\vec{n}$:
\[ W^s(\Gr_J(\vec{n}; V)) = \Gr(\vec{n}; V). \]
\end{theorem}

\begin{proof}
If $P \in W^s(\Gr_J(\vec{n}; V))$, then by definition $n_0(P_\infty) = n_0$. By Lemma~\ref{lem:type_preservation}, we have $n_0(P) = n_0$. So $P \in \Gr(\vec{n}; V)$.
Conversely, if $P \in \Gr(\vec{n}; V)$, then by definition $n_0(P) = n_0$. By Lemma~\ref{lem:type_preservation} $n_0(P_\infty) = n_0$. So $P \in W^s(\Gr_J(\vec{n}; V))$.
\end{proof}

\begin{corollary}\label{cor:def_retraction}
    The Grassmannian $\Gr(\vec{n}; V)$ is homotopy equivalent to the compact homogeneous space $\U(n)/(\O(n_0) \times \U(n_+) \times \U(n_-))$.
\end{corollary}

\begin{proof}
By Theorem \ref{thm:stable_orbits}, the symplectic orbit $\Gr(\vec{n}; V)$ coincides with the stable manifold of the critical submanifold $\Gr_J(\vec{n}; V)$. Since the energy function $f$ is Morse-Bott (Theorem \ref{thm:morse_bott}), the negative gradient flow defines a deformation retraction of the stable manifold onto the critical submanifold. Thus, we have a homotopy equivalence:
\[
\Gr(\vec{n}; V) \simeq \Gr_J(\vec{n}; V).
\]
By Theorem \ref{thm:Jcompatible_darboux}, the unitary group $\U(V, J) \cong \U(n)$ acts transitively on $\Gr_J(\vec{n}; V)$. As shown in Section \ref{subsec:unitary}, the stabilizer of a $J$-compatible subspace is isomorphic to $\O(n_0) \times \U(n_+) \times \U(n_-)$. Therefore, the critical submanifold is diffeomorphic to the stated homogeneous space.
\end{proof}

\printbibliography
\end{document}